\documentclass[a4paper,12 pt]{article}
\usepackage[utf8]{inputenc}
\usepackage{amsmath}
\usepackage{amssymb}
\usepackage{latexsym}
\usepackage{amsthm}
\usepackage{amsfonts}
\usepackage{graphicx}
\usepackage{color}
\usepackage{shuffle}
\usepackage{graphicx}
\usepackage{xcolor}

\newtheorem{theorem}{Theorem}
\newtheorem{lemma}{Lemma}
\newtheorem{proposition}{Proposition}

\newtheorem{corollary}{Corollary}
\newtheorem{remark}{Remark}

\def \a{\alpha}
\def \z{\zeta}
\def\k{\textbf{k}}
\def\p{\textbf{p}}
\def\N{\mathbb{N}}

\def\s{\overset{t}{\ast}}
\def\H{\mathfrak{H}_t^1}

\def\sh{\ast}

\def\sh{\ast}

\def\h{\mathfrak{H}}
\def\hh{\mathfrak{h}}

\usepackage[affil-it]{authblk}
\usepackage[a4paper,top=3cm,bottom=2cm,left=1.5cm,right=1.5cm,marginparwidth=1cm]{geometry}
\title{On a t-stuffle product formula for interpolated multiple zeta values}
\author{$\text{Pitu Sarkar}^a, \text{Nita Tamang}^b$ \\ $\text{Department of Mathematics, University of North Bengal}^{a,b}$\\
	West Bengal, India 734013\\ $\text{pitucob2016@gmail.com}^a$,~$\text{nita\_math@nbu.ac.in}^b$  }

\date{}
\usepackage[a4paper,top=3cm,bottom=2cm,left=1.5cm,right=1.5cm,marginparwidth=1cm]{geometry}

\begin{document}
		\maketitle
		\begin{abstract}
			In this paper, we obtain a restricted decomposition formula for interpolated multiple zeta values using t-stuffle product. We then derive a recursive formula of t-stuffle product, which also provides a route to the same formula. In both cases, combinatorial description of t-stuffle product is our basic tool. We also provide alternative proofs by mathematical induction for some of the results. 
	   \end{abstract}

	\section{Introduction}

Multiple Zeta Values (MZVs) are a fascinating class of mathematical objects that have garnered significant attention in recent decades. These intricate numbers have far-reaching connections to various areas of mathematics and physics, including number theory, algebraic geometry, and quantum field theory. The study of MZVs has led to numerous breakthroughs, from the discovery of new algebraic relations to the development of novel computational methods. Also, they are related to Riemann surfaces, which provide a geometric framework for understanding complex analysis. Recently, the concept of Interpolated Multiple Zeta Values has emerged, which are a generalization of MZVs, offering a more nuanced understanding of MZVs and enabling researchers to uncover new connections and patterns in mathematics and physics. This generalization has opened up new avenues for exploration, but it also presents challenges in developing effective computational tools.
	
	
	
	\vspace{.2em}
	
	A finite sequence of positive integers is called an index. An index $\k=\left(k_1,\dots,k_n\right)$ is called an admissible index if $k_1 \geq 2.$ Throughout this article, $\k=\left(k_1,\dots,k_n\right)$ denotes an admissible index. The interpolated multiple zeta value (t-MZV)  introduced by  S. Yamamoto \cite{2}, is a polynomial defined by 
	$$\zeta^{t}(\k)= \zeta^{t}\left( k_1,\dots, k_n\right)=\sum_{\substack{\p=\left(k_1\Box k_2\Box \dots \Box k_n\right)\\ \Box=``," \text{or} ``+"}}t^{n-\text{dep}(\p)}\zeta(\p)~(\in{\mathbb{R}[t]}),$$
whose depth is given by $\text{dep}(\k)=n$ and the quantity $\zeta(\p)$ in the sum is called the multiple zeta value, which is defined by the convergent series
	\begin{align}
		\zeta(\k)= \zeta \left( k_1,\dots,k_n\right)=\sum_{m_1>\dots >m_n\geq 1}\dfrac{1}{m_1^{k_1}\dots m_n^{k_n}}.
	\end{align}
For $t=1$, the interpolated multiple zeta value reduces to a variant of multiple zeta value, called the multiple zeta-star value, defined as follows:
\begin{align}
	\zeta^{\star}(\k)= \zeta^{\star}\left( k_1,\dots,k_n \right)=\sum_{m_1\geq\dots\geq m_n\geq 1}\dfrac{1}{m_1^{k_1}\dots m_n^{k_n}}.
\end{align}

	The study of t-MZVs is an active area of research, where one of the fundamental goals is to understand all algebraic relations they satisfy and to explore the implications of these values in various mathematical and physical contexts.
	However, the algebraic structure of these values is still largely unexplored. Here we consider an algebraic structure related to t-MZVs \cite{6}.
	Assume that $A=\{x, y\}$ is a set of two noncommutative letters. Define $A^*$ as the set of all words that $A$ generates, encompassing the empty word $1$. Let $\hh_t =\mathbb{Q}[t]<A>$ be the noncommutative polynomial algebra over $\mathbb{Q}[t]$ generated by $A$. Define two subalgebras of $\hh_t$,~$$\hh_t^1=\mathbb{Q}[t]+\hh_ty,~ \hh_t^0= \mathbb{Q}[t]+x\hh_ty.$$
	Given any $k\in {\N}$, define $z_k=x^{k-1}y$. For $t=0$, $\hh_0, \hh_0^1$, and $\hh_0^0$ are respectively denoted by $ \hh, \hh^1$, and $ \hh^0$.\\
\textbf{t-Stuffle Product :}
			 The t-harmonic shuffle (stuffle) product $\s$ on $\H$ is a $\mathbb{Q}[t]$-bilinear map, which satisfies the rules\\
			(i)~ $1\s w = w\s1=w$,\\
			(ii)~$z_kw_1\s z_lw_2 = z_k(w_1\s z_lw_2) + z_l(z_kw_1\s w_2)+(1-2t)z_{k+l}(w_1\s w_2) 
			+ [1-\delta(w_1)\delta(w_2)](t^2-t)x^{k+l}(w_1\s w_2), $

			where $w, w_1, w_2 \in{A^*}\cap{\H}$ and $k,l \in{\N}$ and the map $\delta :A^* \rightarrow \{0,1\}$ is defined by 
		$$\delta(w) =
	\begin{cases}
	1& \text{if $w=1$},\\
	0& \text{if $w\neq{1}$}
\end{cases}$$ 
For $t=0$, $\overset{0}{\ast}$ is denoted by $\ast$, which is the stuffle product of multiple zeta values (see \cite{4}). 	
Now, define a $\mathbb{Q}[t]$ linear map $S_t:\hh_t \rightarrow \hh_t$, by $S_t(1)=1$ and 
		$$S_t(wa)=\sigma_t(w)a,~\text{(for all } w\in{A^*,\text{for all }a \in{A})},$$
where  $\sigma_t$ is an automorphism of the noncommutative algebra $\hh_t$ such that $$\sigma_t(x) = x,~\sigma_t(y)=tx+y.$$
For a $\mathbb{Q}$-linear map $Z:\hh^0 \rightarrow \mathbb{R}$, define $$Z^t=Z \circ S_t : \hh_t^0 \rightarrow \mathbb{R}[t]$$ which is a $\mathbb{Q}[t]$-linear map. 
 The interpolated multiple zeta value associated with the map $Z^t$ is defined by 
		\begin{align*}
			{\z^t(k_1,\dots,k_n)=Z^t(z_{k_1}\dots z_{k_n})}.
		\end{align*}
   In \cite{11}, L. Euler obtained a decomposition formula which expressed the product $ \z(m)\z(n), m,n \geq 2$, in terms of double zeta values:	
			\begin{align*}
				\z(m)\z(n)&= \sum_{k=1}^{m} \binom{m+n-k-1}{n-1}\z(m+n-k,k) \nonumber\\
				& \ \ \ + \sum_{k=1}^{n} \binom{m+n-k-1}{m-1}\z(m+n-k,k).
			\end{align*}

Since then, different versions of the above formula have been obtained in various scenarios \cite{1, 8, 9, 10}. Our main result is also a kind of extension of the above formula. The main theorems of this paper are stated bellow:
	
	\begin{theorem}
		\label{thm2}
		For integers $m,u \geq 2, p \geq 1 \text{ and } n,v \geq 0$, 
		\begin{align}
			\label{equn11}
			&z_m{z_p}^n \s z_u{z_p}^v \nonumber \\
			&=\sum_{\substack{1 \leq l \leq n;\\0 \leq k \leq min\{v,n-l\};\\ i+j=k; i,j \in \mathbb{Z}_{\geq 0}\\a_1 + \dots +a_{v+n-l-i-k}=(n+v-l)p;\\a_l=rp;r \in \mathbb{N}, |A|=j}} \binom{v+n-l-2k}{v-k}{(t^2-t)}^i{(1-2t)}^j \times \nonumber \\ 
			& \ \ \ \ \bigg[z_m\{{z_p}^lz_u+(1-2t){z_p}^{l-1}z_{u+p}  +(1-\delta_{v,0}\delta_{n,l})(t^2-t){z_p}^{l-1}{x}^{u+p}\}\bigg] z_{a_1} \dots z_{a_{v+n-l-i-k}} \nonumber \\
			& \ \  +\sum_{\substack{1 \leq l \leq v;\\0 \leq k \leq min\{n,v-l\};\\ i+j=k; i,j \in \mathbb{Z}_{\geq 0}\\a_1 + \dots +a_{v+n-l-i-k}=(n+v-l)p;\\a_l=rp;r \in \mathbb{N}, |A|=j}} \binom{v+n-l-2k}{n-k}{(t^2-t)}^i{(1-2t)}^j \times \nonumber \\ 
			& \ \ \ \ \bigg[z_u\{{z_p}^lz_m+(1-2t){z_p}^{l-1}z_{m+p}  +(1-\delta_{n,0}\delta_{v,l})(t^2-t){z_p}^{l-1}{x}^{m+p}\}\bigg] z_{a_1} \dots z_{a_{v+n-l-i-k}} \nonumber \\
			& \ \  + \sum_{\substack{\\0 \leq k \leq min\{n,v\};\\ i+j=k; i,j \in \mathbb{Z}_{\geq 0}\\a_1 + \dots +a_{v+n-i-k}=(n+v)p;\\a_l=rp;r \in \mathbb{N}, |A|=j}} \binom{v+n-2k}{n-k}{(t^2-t)}^i{(1-2t)}^j \times \nonumber \\ 
			& \ \ \ \ \bigg[z_mz_u+z_uz_m+(1-2t)z_{m+u} +(1-\delta_{n,0}\delta_{v,0})(t^2-t){x}^{m+u}\}\bigg] z_{a_1} \dots z_{a_{v+n-i-k}}. 
		\end{align}
	\end{theorem}

	\begin{theorem}
	\label{thmp3}
	For integers $m, u, p \geq 1 \text{ and } n,v \geq 0$, 
	\begin{align}
		\label{eq6}
		&z_m{z_p}^n \s z_u{z_p}^v \nonumber \\
		&=\sum_{i=1}^{n}z_m\{{z_p}^iz_u+(1-2t){z_p}^{i-1}z_{u+p}+(1-\delta_{v,0}\delta_{n,i})(t^2-t){z_p}^{i-1}x^{u+p}\}({z_p}^{v} \s {z_p}^{n-i}) \nonumber \\
		& \ \ \ + \sum_{i=1}^{v}z_u\{{z_p}^iz_m+(1-2t){z_p}^{i-1}z_{m+p}+(1-\delta_{n,0}\delta_{v,i})(t^2-t){z_p}^{i-1}x^{m+p}\}({z_p}^{n} \s {z_p}^{v-i}) \nonumber \\
		& \ \ \  +\{z_mz_u+z_uz_m+(1-2t)z_{m+u}+(1-\delta_{n,0}\delta_{v,0})(t^2-t)x^{m+u}\}({z_p}^{n} \s {z_p}^{v}).
	\end{align}
\end{theorem}
		
One can observe that both of the above two theorems give the formulas for the same t-stuffle product. Theorem $1$ is proved in Section $2$. In Section $3$, the recursive formula of t-stuffle product is derived first, which is then followed by the proof of Theorem $2$. It is also shown that Theorem $1$ can be obtained from Theorem $2$. Applying the $\mathbb{Q}[t]$-linear map $Z^t$ on both sides of \eqref{equn11}, we obtain the following restricted decomposition formula for interpolated multiple zeta values:

\begin{theorem}
	\label{thmp4}
	Let  $m,u \geq 2, p \geq 1, n,v \geq 0$ are integers. Then we have
	\begin{align}
		\label{equn12}
		&\z^t (m,\underbrace{p, \dots, p}_n) \z^t (u,\underbrace{p, \dots, p}_v)\nonumber \\
		&=\sum_{\substack{1 \leq l \leq n;\\0 \leq k \leq min\{v,n-l\};\\ i+j=k; i,j \in \mathbb{Z}_{\geq 0}\\a_1 + \dots +a_{v+n-l-i-k}=(n+v-l)p;\\a_l=rp;r \in \mathbb{N}, |A|=j}} \binom{v+n-l-2k}{v-k}{(t^2-t)}^i{(1-2t)}^j \times \nonumber \\ 
		& \ \ \ \  \bigg[\z^t(m,{\{p\}}^l,u,a_1,  \dots , a_{v+n-l-i-k})+(1-2t)\z^t(m,\{p\}^{l-1},u+p,a_1,  \dots , a_{v+n-l-i-k})\nonumber \\ 
		& \ \ \ \  +(1-\delta_{v,0}\delta_{n,l})(t^2-t)\z^t(m,\{p\}^{l-1},u+p+a_1, a_2, \dots , a_{v+n-l-i-k})\bigg]  \nonumber \\
		&+\sum_{\substack{1 \leq l \leq v;\\0 \leq k \leq min\{n,v-l\};\\ i+j=k; i,j \in \mathbb{Z}_{\geq 0}\\a_1 + \dots +a_{v+n-l-i-k}=(n+v-l)p;\\a_l=rp;r \in \mathbb{N}, |A|=j}} \binom{v+n-l-2k}{n-k}{(t^2-t)}^i{(1-2t)}^j \times \nonumber \\ 
		& \ \ \ \  \bigg[\z^t(u,\{p\}^l,m,a_1,  \dots , a_{v+n-l-i-k})+(1-2t)\z^t(u,\{p\}^{l-1},m+p,a_1,  \dots , a_{v+n-l-i-k})\nonumber \\ 
		& \ \ \ \  +(1-\delta_{n,0}\delta_{v,l})(t^2-t)\z^t(u,\{p\}^{l-1},m+p+a_1, a_2, \dots , a_{v+n-l-i-k})\bigg] 
		\nonumber \\
		&+ \sum_{\substack{\\0 \leq k \leq min\{n,v\};\\ i+j=k; i,j \in \mathbb{Z}_{\geq 0}\\a_1 + \dots +a_{v+n-i-k}=(n+v)p;\\a_l=rp;r \in \mathbb{N}, |A|=j}} \binom{v+n-2k}{n-k}{(t^2-t)}^i{(1-2t)}^j \times \bigg[\z^t(m,u,a_1, a_2, \dots , a_{v+n-i-k}) \nonumber \\ 
		& \ \ \ \  +\z^t(u,m,a_1, a_2, \dots , a_{v+n-i-k}) +(1-2t)\z^t(m+u,a_1, a_2, \dots , a_{v+n-i-k})  \nonumber \\
		& \ \ \ \  +(1-\delta_{n,0}\delta_{v,0})(t^2-t)\z^t(m+u+a_1, a_2, \dots , a_{v+n-i-k})\bigg]. 
	\end{align}
\end{theorem}
The above formula generalizes the one for multiple zeta values given by K. W. Chen\cite[Theorem 1]{7} and for multiple zeta-star values given by Z. Li and C. Qin\cite[Corollary 1.5]{3}. 	Taking $p=1$ in \eqref{equn12}, we get the following decomposition formula for t-MZVs of height one, that is, for t-MZVs with all the arguments $1$ except the first:
	\begin{corollary}
		Let  $m,u \geq 2, n,v \geq 0$ are integers. Then we have
		\begin{align*}
			&\z^t (m,\underbrace{1, \dots, 1}_n) \z^t (u,\underbrace{1, \dots, 1}_v)\nonumber \\
			&=\sum_{\substack{1 \leq l \leq n;\\0 \leq k \leq min\{v,n-l\};\\ i+j=k; i,j \in \mathbb{Z}_{\geq 0}\\a_1 + \dots +a_{v+n-l-i-k}=(n+v-l);\\a_l \in \mathbb{N}, |A|=j}} \binom{v+n-l-2k}{v-k}{(t^2-t)}^i{(1-2t)}^j \times \nonumber \\ 
			& \ \ \ \  \bigg[\z^t(m,{\{1\}}^l,u,a_1,  \dots , a_{v+n-l-i-k})+(1-2t)\z^t(m,\{1\}^{l-1},u+1,a_1,  \dots , a_{v+n-l-i-k})\nonumber \\
		\end{align*}
	\begin{align*} 
			& \ \ \ \  +(1-\delta_{v,0}\delta_{n,l})(t^2-t)\z^t(m,\{1\}^{l-1},u+1+a_1, a_2, \dots , a_{v+n-l-i-k})\bigg]  \nonumber \\
			&+\sum_{\substack{1 \leq l \leq v;\\0 \leq k \leq min\{n,v-l\};\\ i+j=k; i,j \in \mathbb{Z}_{\geq 0}\\a_1 + \dots +a_{v+n-l-i-k}=(n+v-l);\\a_l \in \mathbb{N}, |A|=j}} \binom{v+n-l-2k}{n-k}{(t^2-t)}^i{(1-2t)}^j \times \nonumber \\ 
			& \ \ \ \  \bigg[\z^t(u,\{1\}^l,m,a_1,  \dots , a_{v+n-l-i-k})+(1-2t)\z^t(u,\{1\}^{l-1},m+1,a_1,  \dots , a_{v+n-l-i-k})\nonumber \\ 
			& \ \ \ \  +(1-\delta_{n,0}\delta_{v,l})(t^2-t)\z^t(u,\{1\}^{l-1},m+1+a_1, a_2, \dots , a_{v+n-l-i-k})\bigg] 
			\nonumber \\
			&+ \sum_{\substack{\\0 \leq k \leq min\{n,v\};\\ i+j=k; i,j \in \mathbb{Z}_{\geq 0}\\a_1 + \dots +a_{v+n-i-k}=(n+v);\\a_l \in \mathbb{N}, |A|=j}} \binom{v+n-2k}{n-k}{(t^2-t)}^i{(1-2t)}^j \times \bigg[\z^t(m,u,a_1, a_2, \dots , a_{v+n-i-k}) \nonumber \\ 
			& \ \ \ \  +\z^t(u,m,a_1, a_2, \dots , a_{v+n-i-k}) +(1-2t)\z^t(m+u,a_1, a_2, \dots , a_{v+n-i-k})  \nonumber \\
			& \ \ \ \  +(1-\delta_{n,0}\delta_{v,0})(t^2-t)\z^t(m+u+a_1, a_2, \dots , a_{v+n-i-k})\bigg]. 
		\end{align*}	
	\end{corollary}

	\section{Proof of Theorem $1$}
	
	We first define the combinatorial description of t-stuffle product.
Let $\mathbb{A}=\{k_1, k_2, \dots , k_n\}$ and $\mathbb{B}=\{l_1, l_2, \dots , l_m\}$ be two set of positive integers. Then 
	\begin{align}
		\label{p}
		& z_{k_1}z_{k_2} \dots z_{k_n} \s 	z_{l_1}z_{l_2} \dots z_{l_m} \nonumber \\
		&=\sum_{\substack{0 \leq k \leq \text{min}\{m,n\};\\ i+j=k,i,j \in \mathbb{Z}_{\geq 0}}} {(1-2t)}^i {(t^2-t)}^j \sum_{\sigma \in {\mathbb{P}}_{n,m,k,j}} z_{{\sigma}^{-1}(1)} z_{{\sigma}^{-1}(2)} \dots z_{{\sigma}^{-1}(n+m-k-j)},
	\end{align}
	where ${\mathbb{P}}_{n,m,k,j}$ is the set of all surjections 
	$$\sigma : \{1,2, \dots , n+m\} \rightarrow \{1,2, \dots , n+m-k-j\}$$
	with conditions
	\begin{align*}
		&(i) ~\sigma (1) \leq \sigma (2) \leq \dots \leq \sigma (n); \sigma (n+1) \leq \sigma (n+2) \leq \dots  \leq \sigma (n+m)\\
		&(ii)~ \text{For any } \sigma \in {\mathbb{P}}_{n,m,k,j} ~\text{and for} ~r \in \{1,2 , \dots ,n+m-k-j\},
	\end{align*}
	
	$$ z_{{\sigma}^{-1}(r)}=
	\begin{cases}
		z_s& \text{if $ {\sigma}^{-1}(r)=\{s\}$}\\
		z_{s_1+ \dots + s_u+{s'
			}_1+ \dots+ {s'}_{u'}}& \text{if ${\sigma^{-1}(r)=\{s_1, \dots , s_u, {s'}_1, \dots , {s'}_{u'}\}}$}
	\end{cases}
	$$
	where 			
	$$	\{s_1, \dots , s_u\} \subseteq \mathbb{A} ~\text{and}~ \{{s'}_1, \dots , {s'}_{u'}\} \subseteq \mathbb{B},~ u,u' \geq 1$$
	\begin{align*}
	\hspace*{1.6cm}	&(iii) ~\text{In } {{\sigma}^{-1}(r)}, |u-u'| \leq 1\\
		&(iv)~ \text{In each term, number of}~ z_{{\sigma}^{-1}(r)}, \text{where}~ u=u', ~\text{is given by the exponent} \\
			& \ \ \ \ \ \  \text{of}~ (1-2t).
	\end{align*}

	We now derive the following general t-stuffle product formula.
	\begin{proposition}
		\label{them1}
		For positive integers $m, n, k_{i'}, l_{j'}$, where $i'=1, 2, ...., n$ and $j'=1, 2, ...., m$, we have   
		\begin{align}
			\label{equain1}
			& z_{k_1}z_{k_2} \dots z_{k_n} \s 	z_{l_1}z_{l_2} \dots z_{l_m} \nonumber \\
			&=\sum_{\substack{0 \leq k \leq \text{min}\{m,n\};\\ i+j=k,i,j \in \mathbb{Z}_{\geq 0}}} {(1-2t)}^i {(t^2-t)}^j \sum_{\substack{a_1+ \dots + a_{n+m-k-j}\\ =\sum_{r=1}^{n} k_r+\sum_{s=1}^{m} l_s,\\
					{a_{w}}\in \{{k_r}, {l_s}, {s_1+ \dots + s_u+{s'
						}_1+ \dots +{s'}_{u'}}:\\ r=1, \dots , n; s=1, \dots , m;|u-u'| \leq 1\},\\ \text{For}~v\in\{k,l\},~ z_{v_b+\alpha} \text{ cannot be placed before}\\ z_{v_a+\beta} ~\text{if}~ a < b, \text{ where }\alpha, \beta \geq 0; \\ \text{ Number of}~ {a_w}\text{'s},~ \text{where}~ u=u' ~\text{is}~ i }} z_{a_1} z_{a_2} \dots z_{a_{n+m-k-j}},
		\end{align}
		where
		$$	\{s_1, \dots , s_u\} \subseteq \{k_1, k_2, \dots , k_n\} ~\text{and}~ \{{s'}_1, \dots , {s'}_{u'}\} \subseteq \{l_1, l_2, \dots , l_m\}, u,u' \geq 1.$$
	\end{proposition}
	
	\begin{proof}
		In \eqref{p}, denote ${\sigma}^{-1}(w)$ by $a_w$ for any $w \in \{1,2 , \dots ,n+m-k-j\}$. Then it is easy to see that $a_1+ \dots + a_{n+m-k-j} =\sum_{r=1}^{n} k_r+\sum_{s=1}^{m} l_s$. By the condition (i) of \eqref{p}, we can see that 
		$z_{k_1}, \dots	, z_{k_n}$ preserves the relative order, i.e., for $a, b \in \{1, 2, \dots , n\}$ with $a < b$, $z_{k_b+\alpha}$ cannot be placed before $z_{k_a+\beta}$, where $\alpha, \beta \geq 0$. Similarly, for $a', b' \in \{1, 2, \dots , m\}$ with $a' < b'$, $z_{l_{b'}+\alpha}$ cannot be placed before $z_{l_{a'}+\beta}$. By the condition (ii) and (iii), for any $w \in \{1, 2, \dots ,n+m-k-j\}$, $a_w \in  \{{k_1}, \dots , k_n, {l_1}, \dots , l_m, {s_1+ \dots + s_u+{s'}_1+ \dots +{s'}_{u'}}\},$ where $	\{s_1, \dots , s_u\} \subseteq \{k_1, k_2, \dots , k_n\}$ and $\{{s'}_1, \dots , {s'}_{u'}\} \subseteq \{l_1, l_2, \dots , l_m\},$  with $u,u' \geq 1$ and $|u-u'| \leq 1.$
		Also, in each term, number of $a_w$'s where $u=u'$ is $i$. Hence the proof. 
	\end{proof}
	\textbf{Alternative proof:}
	Here we use induction on $m+n$. Assume that \eqref{equain1} holds for $u+v <m+n$. For both $m=1$ and $n=1$ it is obvious. Let one of them be $> 1$.
	 For positive integers $f, g, h$, define 
	\begin{align*}
		\mathbb{S}_f^{g,h}=\sum_{\substack{a_1+ \dots + a_{f}\\ =\sum_{r=g}^{n} k_r+\sum_{s=h}^{m} l_s\\
				{a_{w}}\in \{{k_r}, {l_s}, {s_1+ \dots + s_u+{s'
					}_1+ \dots +{s'}_{u'}}:\\ r=g, \dots , n; s=h, \dots , m;|u-u'| \leq 1\},\\
				\{s_1, \dots , s_u\} \subseteq \{ k_g, \dots , k_n\}, \{{s'}_1, \dots , {s'}_{u'}\} \subseteq \{l_h, \dots , l_m\};\\ \text{For}~v\in\{k,l\},~ z_{v_b+\alpha} \text{ cannot be placed before}\\ z_{v_a+\beta} ~\text{if}~ a < b, \text{ where }\alpha, \beta \geq 0; \\  \text{ Number of}~ {a_w}'s, \text{where}~ u=u', ~\text{is the exponent of}~ (1-2t)}} z_{a_1} z_{a_2} \dots z_{a_{f}}.
	\end{align*}
	By the definition of t-stuffle product, we have
	\begin{align*}
		& z_{k_1}z_{k_2} \dots z_{k_n} \s 	z_{l_1}z_{l_2} \dots z_{l_m}\\
		&=z_{k_1}(z_{k_2} \dots z_{k_n} \s 	z_{l_1}z_{l_2} \dots z_{l_m})+z_{l_1}(z_{k_1}z_{k_2} \dots z_{k_n} \s 	z_{l_2} \dots z_{l_m}) \nonumber \\
		& \ \ +\{(1-2t)z_{k_1+l_1}+(t^2-t)x^{k_1+l_1}\}(z_{k_2} \dots z_{k_n} \s 	z_{l_2} \dots z_{l_m}).
	\end{align*}
	Total number of $z_k$'s in all of the products $(z_{k_2} \dots z_{k_n} \s 	z_{l_1}z_{l_2} \dots z_{l_m}), (z_{k_1}z_{k_2} \dots z_{k_n} \s 	z_{l_2} \dots z_{l_m})$ and \\ $ (z_{k_2} \dots z_{k_n} \s 	z_{l_2} \dots z_{l_m})$ are less than $m+n$. Using inductive hypothesis in the above expression, we get
	\begin{align}
		\label{equ2}
		&\hspace{1em} z_{k_1}z_{k_2} \dots z_{k_n} \s 	z_{l_1}z_{l_2} \dots z_{l_m} \nonumber \\
		&  =z_{k_1}\sum_{\substack{0 \leq k \leq \text{min}\{m,n-1\};\\ i+j=k,i,j \in \mathbb{Z}_{\geq 0}}} {(1-2t)}^i {(t^2-t)}^j~\mathbb{S}_{n+m-k-j-1}^{2,1} + z_{l_1}\sum_{\substack{0 \leq k \leq\text{min}\{m-1,n\};\\ i+j=k,i,j \in \mathbb{Z}_{\geq 0}}} {(1-2t)}^i {(t^2-t)}^j~\mathbb{S}_{n+m-k-j-1}^{1,2} \nonumber \\
		&\ \ \ +\{(1-2t)z_{k_1+l_1}+(t^2-t)x^{k_1+l_1}\}  \sum_{\substack{0 \leq k \leq \text{min}\{m-1,n-1\};\\ i+j=k,i,j \in \mathbb{Z}_{\geq 0}}} {(1-2t)}^i {(t^2-t)}^j~\mathbb{S}_{n+m-k-j-2}^{2,2}. 
	\end{align}
Since the stuffle product $\s$ is commutative, we can assume that $m \geq n$. If $m=n$, then \eqref{equ2} becomes
	\begin{align*}
		&\hspace{1em} z_{k_1}z_{k_2} \dots z_{k_n} \s 	z_{l_1}z_{l_2} \dots z_{l_n} \nonumber \\
		&=\sum_{\substack{0 \leq k \leq n-1;\\ i+j=k,i,j \in \mathbb{Z}_{\geq 0}}} {(1-2t)}^i {(t^2-t)}^j\Bigg[ z_{k_1} \mathbb{S}_{2n-k-j-1}^{2,1}+ z_{l_1}\mathbb{S}_{2n-k-j-1}^{1,2}
		 \Bigg]\nonumber \\
		 & \ \ \ \ + \{(1-2t)z_{k_1+l_1}+(t^2-t)x^{k_1+l_1}\}  \sum_{\substack{1\leq k \leq n;\\ i+j=k-1,i,j \in \mathbb{Z}_{\geq 0}}} {(1-2t)}^i {(t^2-t)}^j~\mathbb{S}_{2n-k-j-1}^{2,2}
	\end{align*}
	\begin{align}
		\label{equ3}
		& =\sum_{\substack{a_1+ \dots + a_{2n}\\ =\sum_{r=1}^{n} k_r+\sum_{s=1}^{n} l_s;\\
				{a_{w}}\in \{{k_r}, {l_s}: r=1, \dots , n; s=1, \dots , n\};\\ \text{For}~ v \in \{k, l\},~ z_{v_b} \text{ placed after}~\\ z_{v_a}~ \text{if}~ a<b  }} z_{a_1} z_{a_2} \dots z_{a_{2n}}
		+ \{(1-2t)z_{k_1+l_1}+(t^2-t)x^{k_1+l_1}\} \sum_{\substack{ i+j=n-1,\\i,j \in \mathbb{Z}_{\geq 0}}} {(1-2t)}^i  \nonumber \\
		& \ \ \ \times   {(t^2-t)}^j ~\mathbb{S}_{n-j-1}^{2,2} + \sum_{\substack{1 \leq k \leq n-1}}\Bigg[ \sum_{\substack{ i+j=k,\\ i,j \in \mathbb{Z}_{\geq 0}}} {(1-2t)}^i {(t^2-t)}^j\Bigg\{ z_{k_1} \mathbb{S}_{2n-k-j-1}^{2,1}+ z_{l_1}\mathbb{S}_{2n-k-j-1}^{1,2}\Bigg\}\nonumber \\
		& \ \ \ + \{(1-2t)z_{k_1+l_1}+(t^2-t)x^{k_1+l_1}\} \sum_{\substack{ i+j=k-1,\\ i,j \in \mathbb{Z}_{\geq 0}}} {(1-2t)}^i {(t^2-t)}^j ~\mathbb{S}_{2n-k-j-1}^{2,2} 
	 \Bigg].
	\end{align}
It is obvious that,
	\begin{align}
		\label{equ4}
		& \sum_{\substack{ i+j=k,\\ i,j \in \mathbb{Z}_{\geq 0}}} {(1-2t)}^i {(t^2-t)}^j\Bigg\{ z_{k_1}\mathbb{S}_{m+n-k-j-1}^{2,1}+ z_{l_1}\mathbb{S}_{m+n-k-j-1}^{1,2}\Bigg\}\nonumber \\
		& \ \ \ + \{(1-2t)z_{k_1+l_1}+(t^2-t)x^{k_1+l_1}\} \sum_{\substack{ i+j=k-1,\\ i,j \in \mathbb{Z}_{\geq 0}}} {(1-2t)}^i {(t^2-t)}^j ~\mathbb{S}_{2n-k-j-1}^{2,2}\nonumber \\ 
		&=\sum_{\substack{ i+j=k,\\ i,j \in \mathbb{Z}_{\geq 0}}} {(1-2t)}^i {(t^2-t)}^j ~\mathbb{S}_{m+n-k-j}^{1,1} . 
	\end{align}
Using \eqref{equ4}  in \eqref{equ3}, in case of $m=n$ and $k=m=n$, we get
	\begin{align*}
		&\hspace{1em} z_{k_1}z_{k_2} \dots z_{k_n} \s 	z_{l_1}z_{l_2} \dots z_{l_n} \nonumber \\
		&=\sum_{\substack{a_1+ \dots + a_{2n}\\ =\sum_{r=1}^{n} k_r+\sum_{s=1}^{n} l_s\\
				{a_{w}}\in \{{k_r}, {l_s}: r=1, \dots , n; s=1, \dots , n\},\\ \text{For}~ v \in \{k, l\},~ z_{v_b} \text{ placed after}~\\ z_{v_a}~ \text{if}~ a<b  }} z_{a_1} z_{a_2} \dots z_{a_{2n}} \nonumber \\
		& \ \ \ +\sum_{\substack{ i+j=n,\\ i,j \in \mathbb{Z}_{\geq 0}}} {(1-2t)}^i {(t^2-t)}^j ~\mathbb{S}_{n-j}^{1,1}  + \sum_{\substack{1 \leq k \leq n-1; \\ i+j=k,\\ i,j \in \mathbb{Z}_{\geq 0}}} {(1-2t)}^i {(t^2-t)}^j ~\mathbb{S}_{2n-k-j}^{1,1}\nonumber \\
		& = \sum_{\substack{0 \leq k \leq n; \\ i+j=k, i,j \in \mathbb{Z}_{\geq 0}}} {(1-2t)}^i {(t^2-t)}^j~ \mathbb{S}_{2n-k-j}^{1,1}.
	\end{align*}
Hence $m=n$ case holds.\\
Let $ m > n$. Then \eqref{equ2} becomes
	\begin{align}
\label{equ8}
		&\hspace{1em} z_{k_1}z_{k_2} \dots z_{k_n} \s 	z_{l_1}z_{l_2} \dots z_{l_m} \nonumber \\
		&  =z_{k_1}\sum_{\substack{0 \leq k \leq n-1;\\ i+j=k,\\ i,j \in \mathbb{Z}_{\geq 0}}} {(1-2t)}^i {(t^2-t)}^j ~\mathbb{S}_{m+n-k-j-1}^{2,1}   + z_{l_1}\sum_{\substack{0 \leq k \leq n;\\ i+j=k,\\ i,j \in \mathbb{Z}_{\geq 0}}} {(1-2t)}^i {(t^2-t)}^j ~\mathbb{S}_{m+n-k-j-1}^{1,2} \nonumber \\
		&\ \ \ +\{(1-2t)z_{k_1+l_1}+(t^2-t)x^{k_1+l_1}\}  \sum_{\substack{0 \leq k \leq n-1;\\ i+j=k,\\ i,j \in \mathbb{Z}_{\geq 0}}} {(1-2t)}^i {(t^2-t)}^j ~\mathbb{S}_{m+n-k-j-2}^{2,2}\nonumber \\
		& = z_{l_1}\sum_{\substack{ i+j=n, \\ i,j \in \mathbb{Z}_{\geq 0}}} {(1-2t)}^i {(t^2-t)}^j ~\mathbb{S}_{m-j-1}^{1,2} \nonumber \\
		& \ \ \ + \sum_{\substack{0 \leq k \leq n-1;\\ i+j=k, \\ i,j \in \mathbb{Z}_{\geq 0}}} {(1-2t)}^i {(t^2-t)}^j \Bigg[z_{k_1} \mathbb{S}_{m+n-k-j-1}^{2,1}
		 +  z_{l_1}~ \mathbb{S}_{m+n-k-j-1}^{1,2} \Bigg]\nonumber \\
		 & \ \ \ + \{(1-2t)z_{k_1+l_1}+(t^2-t)x^{k_1+l_1}\}  \sum_{\substack{1 \leq k \leq n;\\ i+j=k-1,i,j \in \mathbb{Z}_{\geq 0}}} {(1-2t)}^i {(t^2-t)}^j ~\mathbb{S}_{n+m-k-j-1}^{2,2} 
		 \nonumber \\	
		&=\sum_{\substack{a_1+ \dots + a_{m+n}\\ =\sum_{r=1}^{n} k_r+\sum_{s=1}^{m} l_s\\
				{a_{w}}\in \{{k_r}, {l_s}: r=1, \dots , n; s=1, \dots , m\},\\ \text{For}~ v \in \{k, l\},~ z_{v_b} \text{ placed after}~\\ z_{v_a}~ \text{if}~ a<b  }} z_{a_1} z_{a_2} \dots z_{a_{m+n}} + 	z_{l_1}\sum_{\substack{ i+j=n, \\ i,j \in \mathbb{Z}_{\geq 0}}} {(1-2t)}^i {(t^2-t)}^j ~\mathbb{S}_{m-j-1}^{1,2}\nonumber \\ 
		& \ \ + \{(1-2t)z_{k_1+l_1}+(t^2-t)x^{k_1+l_1}\}    \sum_{\substack{ i+j=n-1,\\i,j \in \mathbb{Z}_{\geq 0}}} {(1-2t)}^i {(t^2-t)}^j ~\mathbb{S}_{m-j-1}^{2,2} \nonumber \\
		& \ \ + \sum_{\substack{1 \leq k \leq n-1}}\Bigg[ \sum_{\substack{ i+j=k,\\ i,j \in \mathbb{Z}_{\geq 0}}} {(1-2t)}^i {(t^2-t)}^j\Bigg\{ z_{k_1} \mathbb{S}_{m+n-k-j-1}^{2,1}+z_{l_1}~\mathbb{S}_{m+n-k-j-1}^{1,2}\Bigg\}\nonumber \\
		& \ \  + \{(1-2t)z_{k_1+l_1}+(t^2-t)x^{k_1+l_1}\} \sum_{\substack{ i+j=k-1,\\ i,j \in \mathbb{Z}_{\geq 0}}} {(1-2t)}^i {(t^2-t)}^j~ \mathbb{S}_{m+n-k-j-1}^{2,2} \Bigg].
	\end{align}
Using \eqref{equ4} in \eqref{equ8}, we get
	\begin{align*}
		 &\hspace{1em} z_{k_1}z_{k_2} \dots z_{k_n} \s 	z_{l_1}z_{l_2} \dots z_{l_m} \nonumber \\
		&=\sum_{\substack{a_1+ \dots + a_{m+n}\\ =\sum_{r=1}^{n} k_r+\sum_{s=1}^{m} l_s\\
				{a_{w}}\in \{{k_r}, {l_s}: r=1, \dots , n; s=1, \dots , m\},\\ \text{For}~ v \in \{k, l\},~ z_{v_b} \text{ placed after}~\\ z_{v_a}~ \text{if}~ a<b }} z_{a_1} z_{a_2} \dots z_{a_{m+n}}\nonumber \\
		& \ \ \ + \sum_{\substack{ i+j=n, \\ i,j \in \mathbb{Z}_{\geq 0}}} {(1-2t)}^i {(t^2-t)}^j~
		\mathbb{S}_{m-j}^{1,1}+ \sum_{\substack{1 \leq k \leq n-1}} \sum_{\substack{ i+j=k,\\ i,j \in \mathbb{Z}_{\geq 0}}} {(1-2t)}^i {(t^2-t)}^j ~\mathbb{S}_{m+n-k-j}^{1,1} \nonumber \\
	\end{align*}
\begin{align*}
		& =\sum_{\substack{0 \leq k \leq n}} \sum_{\substack{ i+j=k,\\ i,j \in \mathbb{Z}_{\geq 0}}} {(1-2t)}^i {(t^2-t)}^j ~\mathbb{S}_{m+n-k-j}^{1,1}.
	\end{align*}
This completes the proof.\\

	For $t=0$, we get the following general usual stuffle product formula.
	\begin{corollary}
		For positive integers $m, n, k_i, l_j$, where $i=1, 2, ...., n$ and $j=1, 2, ...., m$, we have   
		\begin{align*}
			& z_{k_1}z_{k_2} \dots z_{k_n} \sh 	z_{l_1}z_{l_2} \dots z_{l_m} \\
			&=\sum_{\substack{0 \leq k \leq \text{min}\{m,n\};\\a_1+ \dots + a_{n+m-k}\\ =\sum_{r=1}^{n} k_r+\sum_{s=1}^{m} l_s,\\
					{a_{w}}\in \{{k_r}, {l_s}, {k_r+l_s} \},\\ \text{For}~v\in\{k,l\},~ z_{v_b+\alpha} \text{ cannot be placed before}\\ z_{v_a+\beta} ~\text{if}~ a < b, \text{ where }\alpha, \beta \geq 0; \\  \text{ Number of}~ z_{k_r+l_s}\text{'s} ~\text{is}~ k }} z_{a_1} z_{a_2} \dots z_{a_{n+m-k}}.
		\end{align*}
	\end{corollary}
	
	\begin{proof}
		If we take $t=0$ in Proposition \ref{them1}, then we have $j=0,i=k$. In this case, there are $n+m-k$ positions for $a_w$ and $k$ numbers of $a_w$ are of the form ${s_1+ \dots + s_u+{s'
			}_1+ \dots +{s'}_{u}}$. Suppose out of $k$ numbers of such $a_w$, for atleast one  $a_w$, $u=2$. Without loss of generality, suppose the first $k-1$ arguments are involved with  $z_{a_{w}}$ where $u=1$ and the next two arguments are involved with $z_{a_{w}}$ where $u=2$. That is $k_1, \dots, k_{k-1}, l_1 , \dots , l_{k-1}$ are involved with $z_{a_{w}}$ where $u=1$, and $k_k, k_{k+1}, l_k, l_{k+1}$ are involved with  $z_{a_{w}}$ where $u=2$. So the total weight of these arguments is $k_1+ \dots + k_{k+1}+l_1+ \dots + l_{k+1}$. The remaining number of positions for $a_w$ is $n+m-2k$ whereas the remaining number of arguments involved is $n-k-1+m-k-1=n+m-2k-2$, which is a contradiction. Hence there are $k$ number of $z_{a_w} $ of the form $z_{k_r+l_s}$.
	\end{proof}
Now we obtain a lemma which plays a pivotal role in proving both Theorem $1$ and Theorem $2$.
	\begin{lemma}
		\label{lem2}
		For integers $m \geq 0, n \geq 0$, and $p \geq 1$, we have 
		\begin{align}
			\label{equn1}
			&\hspace{1em}{z_p}^m \s {z_p}^n \nonumber \\
			& = \sum_{\substack{0 \leq k \leq min\{m,n\};\\ i+j=k; i,j \in \mathbb{Z}_{\geq 0}\\a_1 + \dots +a_{n+m-i-k}=(n+m)p;\\a_l=rp,r \in \mathbb{N};\\ \text{Number of}~ a_l\text{'s } \text{which are even multiple of p is j} }} \binom{m+n-2k}{m-k}{(t^2-t)}^i{(1-2t)}^j z_{a_1} \dots z_{a_{n+m-i-k}}.
		\end{align}
	\end{lemma}
	
	\begin{proof}
		Suppose $|A|$ denotes the number of $a_l$'s which are even multiple of $p$.
		From \eqref{equain1}, we have
		\begin{align} 
			&\hspace{1em}{z_p}^m \s {z_p}^n \nonumber \\
			& = \sum_{\substack{0 \leq k \leq \text{min}\{m,n\};\\ i+j=k; i,j \in \mathbb{Z}_{\geq 0}\\a_1 + \dots +a_{n+m-i-k}=(n+m)p;\\a_l=rp;r \in \mathbb{N}\\ |A|= j }} c_{ij} {(t^2-t)}^i{(1-2t)}^j  z_{a_1} \dots z_{a_{n+m-i-k}}.
		\end{align}
		We only need to determine $c_{ij}$. We can see that for any  $k$, all $c_{ij}$ are equal. When $i=0$, we have $ j=k$ and in the summation, there are $k$ times $z_{2p}$ (since only possible even $a_j$ is $2p$). So there are $(m+n-2k)$ times $z_p$ in the summation in which $(m-k)$ times $z_p$ are from ${z_p}^m$ and others are from  ${z_p}^n$. Hence the coefficient $c_{0k}$ is $ \binom{m+n-2k}{m-k}$. Therefore the coefficient $c_{ij}$ is given by $ \binom{m+n-2k}{m-k}$. This completes the proof.
	\end{proof}
\textbf{Alternative proof:}
	Here we use induction on $m+n$. Assume that \eqref{equn1} holds for $u+v <m+n$. For both $m=1$ and $n=1$ it is obvious. Let one of them be greater than $ 1$.
	By the definition of t-stuffle product, we have
	\begin{align*}
		&{z_p}^m \s {z_p}^n \\
		&=z_p({z_p}^{m-1} \s {z_p}^n)+ z_p({z_p}^{m} \s {z_p}^{n-1})+(1-2t) z_{2p}({z_p}^{m-1} \s {z_p}^{n-1})+(t^2-t) x^{2p}({z_p}^{m-1} \s {z_p}^{n-1}).
	\end{align*}
	The sums of degrees of the factors  ${z_p}^{m-1} \s {z_p}^{n}, {z_p}^{m} \s {z_p}^{n-1} \text{and}~ {z_p}^{m-1} \s {z_p}^{n-1}$ are all less than $m+n$.
	By the inductive hypothesis, we have
	\begin{align}
		\label{equn2}
		&\hspace{1em}{z_p}^m \s {z_p}^n \nonumber \\
		&=z_p \sum_{\substack{0 \leq k \leq min\{m-1,n\};\\ i+j=k; i,j \in \mathbb{Z}_{\geq 0}\\a_1 + \dots +a_{n+m-i-k-1}\\=(n+m-1)p;\\a_l=rp,r \in \mathbb{N}; |A|= j} } \binom{m+n-2k-1}{m-k-1}{(t^2-t)}^i{(1-2t)}^j z_{a_1} \dots z_{a_{n+m-i-k-1}} \nonumber \\
		&\ \ \ +z_p \sum_{\substack{0 \leq k \leq min\{m,n-1\};\\ i+j=k; i,j \in \mathbb{Z}_{\geq 0}\\a_1 + \dots +a_{n+m-i-k-1}\\=(n+m-1)p;\\a_l=rp,r \in \mathbb{N}; |A|= j} } \binom{m+n-2k-1}{m-k}{(t^2-t)}^i{(1-2t)}^j z_{a_1} \dots z_{a_{n+m-i-k-1}} \nonumber \\
		&\ \ \ +\{(1-2t)z_{2p}+(t^2-t)x^{2p}\} \sum_{\substack{0 \leq k \leq min\{m-1,n-1\};\\ i+j=k; i,j \in \mathbb{Z}_{\geq 0}\\a_1 + \dots +a_{n+m-i-k-2}\\=(n+m-2)p;\\a_l=rp,r \in \mathbb{N}; |A|= j} } \binom{m+n-2k-2}{m-k-1}{(t^2-t)}^i{(1-2t)}^j \nonumber \\
		& \hspace{5em} \times z_{a_1} \dots z_{a_{n+m-i-k-2}}. 
	\end{align}
Suppose, for positive integers $d$ and $e$, 
\begin{align*}
	\mathbb{F}_d^e=\sum_{\substack{a_1 + \dots +a_{d}=ep;\\a_l=rp,r \in \mathbb{N}; |A|= j} } z_{a_1} \dots z_{a_{d}}
\end{align*}
Since the stuffle product $\s$ is commutative, we can assume that $m \geq n$. If $m=n$, then \eqref{equn2} becomes
	\begin{align*}
		&\hspace{1em}{z_p}^n \s {z_p}^n \nonumber \\
		&=z_p \sum_{\substack{0 \leq k \leq {n-1};\\ i+j=k; i,j \in \mathbb{Z}_{\geq 0}} } \binom{2n-2k-1}{n-k-1}{(t^2-t)}^i{(1-2t)}^j~	\mathbb{F}_{2n-i-k-1}^{2n-1}  \nonumber \\
		&\ \ \ +z_p \sum_{\substack{0 \leq k \leq {n-1};\\ i+j=k; i,j \in \mathbb{Z}_{\geq 0}} } \binom{2n-2k-1}{n-k}{(t^2-t)}^i{(1-2t)}^j~\mathbb{F}_{2n-i-k-1}^{2n-1} \nonumber \\
		&\ \ \ +\{(1-2t)z_{2p}+(t^2-t)x^{2p}\} \sum_{\substack{0 \leq k \leq {n-1};\\ i+j=k; i,j \in \mathbb{Z}_{\geq 0}} } \binom{2n-2k-2}{n-k-1}{(t^2-t)}^i{(1-2t)}^j~\mathbb{F}_{2n-i-k-2}^{2n-2} \nonumber\\
			\end{align*}
	\begin{align}
	\label{equn3}
		& = z_p \sum_{\substack{0 \leq k \leq {n-1};\\ i+j=k; i,j \in \mathbb{Z}_{\geq 0}} } \binom{2n-2k}{n-k}{(t^2-t)}^i{(1-2t)}^j ~ \mathbb{F}_{2n-i-k-1}^{2n-1} \nonumber \\
		&\ \ \ +\{(1-2t)z_{2p}+(t^2-t)x^{2p}\} \sum_{\substack{1 \leq k \leq {n};\\ i+j=k-1; i,j \in \mathbb{Z}_{\geq 0}} } \binom{2n-2k}{n-k}{(t^2-t)}^i{(1-2t)}^j ~\mathbb{F}_{2n-i-k-1}^{2n-2} \nonumber\\
		& =z_p \binom{2n}{n} {z_p}^{2n-1} + \{(1-2t)z_{2p}+(t^2-t)x^{2p}\}\sum_{\substack{i+j=n-1; i,j \in \mathbb{Z}_{\geq 0}} } {(t^2-t)}^i{(1-2t)}^j \mathbb{F}_{n-i-1}^{2n-2} \nonumber \\
		& \ \ \ +\sum_{1 \leq k \leq {n-1}}\binom{2n-2k}{n-k}\bigg[z_p \sum_{\substack{ i+j=k; i,j \in \mathbb{Z}_{\geq 0}} } {(t^2-t)}^i{(1-2t)}^j~\mathbb{F}_{2n-i-k-1}^{2n-1}\nonumber \\
		&\ \ \ \ \ \hspace{1em} +\{(1-2t)z_{2p}+(t^2-t)x^{2p}\} \sum_{\substack{i+j=k; i,j \in \mathbb{Z}_{\geq 0}} } {(t^2-t)}^i{(1-2t)}^j ~\mathbb{F}_{2n-i-k-1}^{2n-2} \bigg].
	\end{align}
It is easy to see that,
	\begin{align}
		\label{equn4}
		&z_p \sum_{\substack{ i+j=k; i,j \in \mathbb{Z}_{\geq 0}} } {(t^2-t)}^i{(1-2t)}^j~\mathbb{F}_{m+n-i-k-1}^{m+n-1} \nonumber \\
		&\ \  +\{(1-2t)z_{2p}+(t^2-t)x^{2p}\} \sum_{\substack{i+j=k-1; i,j \in \mathbb{Z}_{\geq 0}} } {(t^2-t)}^i{(1-2t)}^j \mathbb{F}_{m+n-i-k-1}^{m+n-2} \nonumber \\
		&= \sum_{\substack{ i+j=k; i,j \in \mathbb{Z}_{\geq 0}} } {(t^2-t)}^i{(1-2t)}^j ~\mathbb{F}_{m+n-i-k}^{m+n}.
	\end{align}
Using \eqref{equn4} in \eqref{equn3}, in case of $m=n $ and $k=m=n$, we get
	\begin{align*}
		&\hspace{1em}{z_p}^n \s {z_p}^n \nonumber \\
		&=\binom{2n}{n} {z_p}^{2n} + \sum_{\substack{i+j=n; i,j \in \mathbb{Z}_{\geq 0}} } {(t^2-t)}^i{(1-2t)}^j~\mathbb{F}_{n-i}^{2n} \nonumber \\
		& \ \ \ \ + \sum_{1 \leq k \leq {n-1}}\binom{2n-2k}{n-k} \sum_{\substack{ i+j=k; i,j \in \mathbb{Z}_{\geq 0}} } {(t^2-t)}^i{(1-2t)}^j~\mathbb{F}_{2n-i-k}^{2n} \nonumber \\
		&=\sum_{0 \leq k \leq {n}}\binom{2n-2k}{n-k} \sum_{\substack{ i+j=k; i,j \in \mathbb{Z}_{\geq 0}} } {(t^2-t)}^i{(1-2t)}^j ~\mathbb{F}_{2n-i-k}^{2n}.	
	\end{align*}
Therefore the case of $m=n$ holds. \\
Let $m > n $. Then \eqref{equn2} becomes
	\begin{align*}
		&\hspace{1em}{z_p}^m \s {z_p}^n \nonumber \\
		&=z_p \sum_{\substack{0 \leq k \leq {n};\\ i+j=k; i,j \in \mathbb{Z}_{\geq 0}} } \binom{m+n-2k-1}{m-k-1}{(t^2-t)}^i{(1-2t)}^j~ \mathbb{F}_{m+n-i-k-1}^{m+n-1} \nonumber \\
		&\ \ \ +z_p \sum_{\substack{0 \leq k \leq {n-1};\\ i+j=k; i,j \in \mathbb{Z}_{\geq 0}} } \binom{m+n-2k-1}{m-k}{(t^2-t)}^i{(1-2t)}^j ~ \mathbb{F}_{m+n-i-k-1}^{n+m-1} \nonumber \\
	\end{align*}
\begin{align}
\label{equn8}
		&\ \ \ +\{(1-2t)z_{2p}+(t^2-t)x^{2p}\} \sum_{\substack{0 \leq k \leq {n-1};\\ i+j=k; i,j \in \mathbb{Z}_{\geq 0}} } \binom{m+n-2k-2}{m-k-1}{(t^2-t)}^i{(1-2t)}^j ~\mathbb{F}_{m+n-i-k-2}^{m+n-2} \nonumber \\
		&=z_p \sum_{\substack{ i+j=n; i,j \in \mathbb{Z}_{\geq 0}} } {(t^2-t)}^i{(1-2t)}^j ~\mathbb{F}_{m-i-1}^{m+n-1}\nonumber \\
		&\ \ \ +z_p \sum_{\substack{0 \leq k \leq {n-1};\\ i+j=k; i,j \in \mathbb{Z}_{\geq 0}} } \binom{m+n-2k}{m-k}{(t^2-t)}^i{(1-2t)}^j \mathbb{F}_{m+n-i-k-1}^{m+n-1} \nonumber \\
		&\ \ \ +\{(1-2t)z_{2p}+(t^2-t)x^{2p}\} \sum_{\substack{1 \leq k \leq {n};\\ i+j=k-1; i,j \in \mathbb{Z}_{\geq 0}} } \binom{m+n-2k}{m-k}{(t^2-t)}^i{(1-2t)}^j~\mathbb{F}_{m+n-i-k-1}^{n+m-2} \nonumber \\
		&= \binom{m+n}{m} {z_p}^{n+m} + \sum_{1 \leq k \leq {n-1}} \binom{m+n-2k}{m-k} 
        \bigg[z_p \sum_{\substack{1 \leq k \leq {n-1};\\ i+j=k; i,j \in \mathbb{Z}_{\geq 0}} } {(t^2-t)}^i{(1-2t)}^j ~\mathbb{F}_{m+n-i-k-1}^{m+n-1} \nonumber \\
		&\ \ \ +\{(1-2t)z_{2p}+(t^2-t)x^{2p}\} \sum_{\substack{1 \leq k \leq {n-1};\\ i+j=k-1; i,j \in \mathbb{Z}_{\geq 0}} } {(t^2-t)}^i{(1-2t)}^j~\mathbb{F}_{m+n-i-k-1}^{m+n-2} \bigg] \nonumber \\
		& \ \ \ + \bigg[z_p \sum_{\substack{ i+j=n; i,j \in \mathbb{Z}_{\geq 0}} } {(t^2-t)}^i{(1-2t)}^j ~\mathbb{F}_{m-i-1}^{m+n-1}  \nonumber \\
		&\ \ \ +\{(1-2t)z_{2p}+(t^2-t)x^{2p}\} \sum_{\substack{ i+j=n-1; i,j \in \mathbb{Z}_{\geq 0}} } \binom{m+n-2k}{m-k}{(t^2-t)}^i{(1-2t)}^j ~\mathbb{F}_{m-i-1}^{m+n-2}\bigg].
	\end{align}
Using \eqref{equn4} in \eqref{equn8}, we get
	\begin{align*}
		&\hspace{1em}{z_p}^m \s {z_p}^n \nonumber \\
		&=\binom{m+n}{m} {z_p}^{m+n}   + \sum_{1 \leq k \leq {n-1}}\binom{m+n-2k}{m-k} \sum_{\substack{ i+j=k; i,j \in \mathbb{Z}_{\geq 0}} } {(t^2-t)}^i{(1-2t)}^j ~ \mathbb{F}_{m+n-i-k}^{m+n} \\
		& \ \ + \sum_{\substack{i+j=n; i,j \in \mathbb{Z}_{\geq 0}} } {(t^2-t)}^i{(1-2t)}^j ~\mathbb{F}_{m-i}^{m+n} \nonumber \\
		&=\sum_{0 \leq k \leq {n}}\binom{m+n-2k}{m-k} \sum_{\substack{ i+j=k; i,j \in \mathbb{Z}_{\geq 0}} } {(t^2-t)}^i{(1-2t)}^j ~\mathbb{F}_{m+n-i-k}^{m+n}.	
	\end{align*}
Hence the proof.
\vspace{.5cm}		
	
	
\begin{lemma}
		\label{lem3}
		For integers $n \geq 2, p \geq 1, k \geq 0 $ and $m \geq 0$, we have
		\begin{align}
			\label{equn9}
			& z_n{z_p}^k \s {z_p}^m \nonumber \\
			& = \sum_{l=0}^{m} \big\{{z_p}^lz_n+(1-2t){z_p}^{l-1}z_{n+p}+(1-\delta_{k,0}\delta_{m,l})(t^2-t){z_p}^{l-1}x^{n+p}\big\} ({z_p}^k \s {z_p}^{m-l}),
		\end{align}
		where $\delta_{i,j}$ is the Kronecker's delta symbol defined as
		$$ \delta_{i,j}=
		\begin{cases}
			1& \text{if $i=j$}\\
			0& otherwise
		\end{cases}
		$$
		and we consider ${z_p}^{-1}=0$.
	\end{lemma}
	
	\begin{proof}
		First we fix the integers $n, k$ and $p$. Then we use mathematical induction on $m$.
		Assume that \eqref{equn9} is true for all integers $s$ with $s<m$.
		By the definition of t-stuffle product,	we have 
		\begin{align}
			\label{equn10}
			& z_n{z_p}^k \s {z_p}^m \nonumber \\
			&=z_n({z_p}^k \s {z_p}^m) + z_p(z_n{z_p}^k \s {z_p}^{m-1})+\{(1-2t)z_{n+p}+ (1-\delta_{k,0}\delta_{m,1})(t^2-t)x^{n+p}\}({z_p}^k \s {z_p}^{m-1}).
		\end{align}
		Using the induction hypothesis,
\begin{align*}
			& z_n{z_p}^k \s {z_p}^m \nonumber \\
			&=z_n({z_p}^k \s {z_p}^m) +\{(1-2t)z_{n+p}+ (1-\delta_{k,0}\delta_{m,1})(t^2-t)x^{n+p}\}({z_p}^k \s {z_p}^{m-1}) \nonumber \\
			&\ \ \ +z_p	\sum_{l=0}^{m-1} \big\{{z_p}^lz_n+(1-2t){z_p}^{l-1}z_{n+p}+(1-\delta_{k,0}\delta_{m,l+1})(t^2-t){z_p}^{l-1}x^{n+p}\big\} ({z_p}^k \s {z_p}^{m-l-1}) \nonumber \\
			& = z_n({z_p}^k \s {z_p}^m) +\{(1-2t)z_{n+p}+ (1-\delta_{k,0}\delta_{m,1})(t^2-t)x^{n+p}\}({z_p}^k \s {z_p}^{m-1})+ z_pz_n({z_p}^k \s {z_p}^{m-1}) \nonumber \\
			&\ \ \ + \sum_{l=1}^{m-1} \big\{{z_p}^{l+1}z_n+(1-2t){z_p}^{l}z_{n+p}+(1-\delta_{k,0}\delta_{m,l+1})(t^2-t){z_p}^{l}x^{n+p}\big\} ({z_p}^k \s {z_p}^{m-l-1}) \nonumber \\
			&= z_n({z_p}^k \s {z_p}^m) +\{z_pz_n+(1-2t)z_{n+p}+ (1-\delta_{k,0}\delta_{m,1})(t^2-t)x^{n+p}\}({z_p}^k \s {z_p}^{m-1}) \nonumber \\
			&\ \ \ + \sum_{l=2}^{m} \big\{{z_p}^{l}z_n+(1-2t){z_p}^{l-1}z_{n+p}+(1-\delta_{k,0}\delta_{m,l})(t^2-t){z_p}^{l-1}x^{n+p}\big\} ({z_p}^k \s {z_p}^{m-l}) \nonumber \\
			&=\sum_{l=0}^{m} \big\{{z_p}^{l}z_n+(1-2t){z_p}^{l-1}z_{n+p}+(1-\delta_{k,0}\delta_{m,l})(t^2-t){z_p}^{l-1}x^{n+p}\big\} ({z_p}^k \s {z_p}^{m-l}).
		\end{align*}
		This completes the proof.
	\end{proof}
 
 Proof of Theorem $1$:
		By the definition of t-stuffle product, we have
		\begin{align}
         \label{twenty}
			&z_m{z_p}^n \s z_u{z_p}^v \nonumber \\
			&=z_m({z_p}^n \s z_u{z_p}^v) + z_u(z_m{z_p}^n \s {z_p}^{v})+\{(1-2t)z_{m+u}+ (1-\delta_{n,0}\delta_{v,0})(t^2-t)x^{m+u}\}({z_p}^n \s {z_p}^{v}).
		\end{align}
Using Lemma \ref{lem2} and Lemma \ref{lem3} in \eqref{twenty} , we can easily obtain \eqref{equn11}.	
\section{Recursive formula for t-stuffle product}
To derive the recursive formula for t-stuffle product, we first define the product $\s_o$ : $\H \times \H$ $\rightarrow$ $\h_t$ which is $\mathbb{Q}[t]$-bilinear, and satisfies the rules
	\begin{align*}
		&(i)~ 1\s_o w = w\s_o1=w,\\
		&(ii)~z_kw_1\s_o z_lw_2 = z_k(w_1\s_o z_lw_2) + z_l(z_kw_1\s_o w_2)   
		+\{(1-2t)z_{k+l}+ (t^2-t)x^{k+l}\}(w_1\s_o w_2) \\
		& \ \ \ \ \text{where}~ w, w_1, w_2 \in{A^*}\cap{\H} ~\text{and}~ k,l \in{\N}.
	\end{align*}
This product $\s_o$ is commutative.\\

The following theorem gives the recursive formula for t-stuffle product.
	
	\begin{proposition}
		\label{thm302}
		Let $ m, n, k_1, \dots , k_m, l_1, \dots , l_n$ be positive integers. Then for any integer j with $1 \leq j \leq m$, we have
		\begin{align}
			\label{eq5}
			& z_{k_1}z_{k_2} \dots z_{k_m} \s 	z_{l_1}z_{l_2} \dots z_{l_n} \nonumber \\
			&= \sum_{i=0}^{n} \bigg[(z_{k_1}z_{k_2} \dots z_{k_{j-1}} \s_o 	z_{l_1}z_{l_2} \dots z_{l_i})z_{k_j}+ (z_{k_1}z_{k_2} \dots z_{k_{j-1}} \s_o 	z_{l_1}z_{l_2} \dots z_{l_{i-1}})\nonumber \\
			& \ \ \ \ \times \{(1-2t)z_{k_j+l_i}+(1-\delta_{i,n}\delta_{j,m})(t^2-t)x^{k_j+l_i}\}\bigg] (z_{k_{j+1}} \dots z_{k_m} \s 	z_{l_{i+1}} \dots z_{l_n}),
		\end{align}
where $ z_{k_i} \dots z_{k_{i-1}}=1$ and  $ z_{k_i} \dots z_{k_{i-2}}=0$ for any integer $i$, $\delta_{ij}$ is Kronecker's delta symbol.	 
		
	\end{proposition}
	
\begin{proof} By the combinatorial description of t-stuffle product,  The left-hand side of \eqref{eq5} is given by the sum
		\begin{align*}
			\sum_{\substack{0 \leq k \leq min\{m,n\};\\ i+j=k,i,j \in \mathbb{Z}_{\geq 0}}} {(1-2t)}^j {(t^2-t)}^i \sum_{\sigma \in {\mathbb{P}}_{n,m,k}} z_{{\sigma}^{-1}(1)} z_{{\sigma}^{-1}(2)} \dots z_{{\sigma}^{-1}(n+m-k)},
		\end{align*}
		where ${\mathbb{P}}_{n,m,k,j}$ is the set of all surjections 
		$$\sigma : \{1,2, \dots , m+n\} \rightarrow \{1,2, \dots , m+n-k\}$$
		with conditions
		\begin{align*}
			&(i)~ \sigma (1) < \sigma (2) < \dots < \sigma (m); \sigma (m+1) < \sigma (m+2) < \dots  < \sigma (m+n)\\
			&(ii)~ \text{For any } \sigma \in {\mathbb{P}}_{n,m,k}~ \text{and for} ~r \in \{1,2 , \dots ,n+m-k\},\\
			& \hspace*{3cm}(a)~ z_{{\sigma}^{-1}(r)}=z_s ~\text{if} ~{\sigma}^{-1}(r)=\{s\}\\
			& \hspace*{3cm}(b)~ z_{{\sigma}^{-1}(r)} \in \{z_{s_1+s_2}, (1-\delta_{s_1,n}\delta_{s_2,m})x^{s_1+s_2}\} ~\text{if}~ {\sigma}^{-1}(r)=\{s_1, s_2\}\\
			&(iii)~ \text{In each term, number of}~ z_{{\sigma}^{-1}(r)}, \text{which is equal to }~ z_{s_1+s_2}  ~\text{is given by the exponent of}~ (1-2t)\\
			& \ \ \ \ \ \  ~ \text{and which is equal to }~ (1-\delta_{s_1,n}\delta_{s_2,m})x^{s_1+s_2} ~\text{is given by the exponent of}~ (t^2-t).
		\end{align*}
		
		As in \cite{3}, we can see that for a fixed $j \in \{1,2, \dots , m\}$, there are two different types of $\sigma^{,}$s one satisfies the condition $ \sigma(i) < \sigma(j) <\sigma(i+1)$ for some $i \in \{m, m+1, \dots , m+n\}$, and the other one satisfies the condition $\sigma(j)=\sigma(i)$ for some $i \in \{m+1, m+2, \dots , m+n\}$. Summing over the first type of $\sigma^{,}$s we get the term 
		
		$$ \sum_{i=0}^{n} (z_{k_1}z_{k_2} \dots z_{k_{j-1}} \s_o 	z_{l_1}z_{l_2} \dots z_{l_i})z_{k_j}(z_{k_{j+1}} \dots z_{k_m} \s 	z_{l_{i+1}} \dots z_{l_n}).$$
Summing over second type of $\sigma^,$s we obtain the term 
		$$\sum_{i=0}^{n} (z_{k_1}z_{k_2} \dots z_{k_{j-1}} \s_o 	z_{l_1}z_{l_2} \dots z_{l_{i-1}})
		\{(1-2t)z_{k_j+l_i}+(1-\delta_{i,n}\delta_{j,m})(t^2-t)x^{k_j+l_i}\} (z_{k_{j+1}} \dots z_{k_m} \s 	z_{l_{i+1}} \dots z_{l_n}).$$
Thus we get the recursive formula.
		\end{proof}
\textbf{Alternative proof:} Here we use  induction on $m+n$. Suppose \eqref{eq5} holds for all $u, v$ with $u+v < m+n$. By the definition of t-stuffle product, 
	\begin{align}
		\label{7}
		&z_{k_1}z_{k_2} \dots z_{k_m} \s 	z_{l_1}z_{l_2} \dots z_{l_n} \nonumber \\
		&=z_{k_1}(z_{k_2} \dots z_{k_m} \s 	z_{l_1}z_{l_2} \dots z_{l_n})+z_{l_1}(z_{k_1}z_{k_2} \dots z_{k_m} \s 	z_{l_2} \dots z_{l_n}) \nonumber \\
		& \ \ +\{(1-2t)z_{k_1+l_1}+(1-\delta_{m,1}\delta_{n,1})(t^2-t)x^{k_1+l_1}\}(z_{k_2} \dots z_{k_m} \s 	z_{l_2} \dots z_{l_n}).
	\end{align}
	Total number of $z_k$'s in each of the t-stuffle products  $(z_{k_2} \dots z_{k_m} \s 	z_{l_1}z_{l_2} \dots z_{l_n}), (z_{k_1}z_{k_2} \dots z_{k_m} \s 	z_{l_2} \dots z_{l_n})$ and $ (z_{k_2} \dots z_{k_m} \s 	z_{l_2} \dots z_{l_n})$ are less than $m+n$. So by using inductive hypothesis in the second term on R.H.S of \eqref{7}, we get
	\begin{align*}
		&z_{k_1}z_{k_2} \dots z_{k_m} \s 	z_{l_1}z_{l_2} \dots z_{l_n}\\
		&=z_{k_1}(z_{k_2} \dots z_{k_m} \s 	z_{l_1}z_{l_2} \dots z_{l_n})+z_{l_1}\bigg[\sum_{i=1}^{n} \bigg[z_{l_2} \dots z_{l_i}z_{k_1}+ (z_{l_2} \dots z_{l_{i-1}})\nonumber \\
		& \ \ \ \ \times \{(1-2t)z_{k_1+l_i}+(1-\delta_{i,n}\delta_{1,m})(t^2-t)x^{k_1+l_i}\}\bigg] (z_{k_{2}} \dots z_{k_m} \s 	z_{l_{i+1}} \dots z_{l_n})\bigg]\\
		& \ \ +\{(1-2t)z_{k_1+l_1}+(1-\delta_{m,1}\delta_{n,1})(t^2-t)x^{k_1+l_1}\}(z_{k_2} \dots z_{k_m} \s 	z_{l_2} \dots z_{l_n}) \\
		&=z_{k_1}(z_{k_2} \dots z_{k_m} \s 	z_{l_1}z_{l_2} \dots z_{l_n})+ z_{l_1}z_{k_1}(z_{k_{2}} \dots z_{k_m} \s 	z_{l_{2}} \dots z_{l_n})\\
		& \ \ + \{(1-2t)z_{k_1+l_1}+(1-\delta_{m,1}\delta_{n,1})(t^2-t)x^{k_1+l_1}\}(z_{k_2} \dots z_{k_m} \s 	z_{l_2} \dots z_{l_n}) \\
		& \ \ + z_{l_1}\bigg[\sum_{i=2}^{n} \bigg[z_{l_2} \dots z_{l_i}z_{k_1}+ z_{l_2} \dots z_{l_{i-1}} \times \{(1-2t)z_{k_1+l_i}+(1-\delta_{i,n}\delta_{1,m})(t^2-t)x^{k_1+l_i}\}\bigg]\nonumber \\
		& \ \ \ \hspace{2em} \times (z_{k_{2}} \dots z_{k_m} \s 	z_{l_{i+1}} \dots z_{l_n})\bigg] \\
		& =z_{k_1}(z_{k_2} \dots z_{k_m} \s 	z_{l_1}z_{l_2} \dots z_{l_n})+ \big[z_{l_1}z_{k_1}+\{(1-2t)z_{k_1+l_1}+(1-\delta_{m,1}\delta_{n,1})(t^2-t)x^{k_1+l_1}\}\big] \\
		& \ \ \times (z_{k_{2}} \dots z_{k_m} \s 	z_{l_{2}} \dots z_{l_n})+ \bigg[\sum_{i=2}^{n} \big\{z_{l_1}z_{l_2} \dots z_{l_i}z_{k_1}+ z_{l_1}z_{l_2} \dots z_{l_{i-1}} \times \\
		& \ \  \{(1-2t)z_{k_1+l_i}+(1-\delta_{i,n}\delta_{1,m})(t^2-t)x^{k_1+l_i}\}\big\}
		\times (z_{k_{2}} \dots z_{k_m} \s 	z_{l_{i+1}} \dots z_{l_n})\bigg]\\ 
		&=\sum_{i=0}^{n} \big\{z_{l_1}z_{l_2} \dots z_{l_i}z_{k_1}+ z_{l_1}z_{l_2} \dots z_{l_{i-1}} \times 
		\{(1-2t)z_{k_1+l_i}+(1-\delta_{i,n}\delta_{1,m})(t^2-t)x^{k_1+l_i}\}\big\}\\
		& \ \hspace{2em} \times (z_{k_{2}} \dots z_{k_m} \s 	z_{l_{i+1}} \dots z_{l_n}).
	\end{align*}
Therefore \eqref{eq5} is true for $j=1$. Let $2 \leq j \leq m$. Now, applying inductive hypothesis in all the three terms on R.H.S of \eqref{7}, we get
	
	\begin{align*}
		&z_{k_1}z_{k_2} \dots z_{k_m} \s 	z_{l_1}z_{l_2} \dots z_{l_n}\\
		&=z_{k_1}\bigg[\sum_{i=0}^{n} \Big\{(z_{k_2} \dots z_{k_{j-1}} \s_o 	z_{l_1}z_{l_2} \dots z_{l_i})z_{k_j}+ (z_{k_2} \dots z_{k_{j-1}} \s_o 	z_{l_1}z_{l_2} \dots z_{l_{i-1}})\nonumber \\
		& \ \ \ \ \times \{(1-2t)z_{k_j+l_i}+(1-\delta_{i,n}\delta_{j,m})(t^2-t)x^{k_j+l_i}\}\Big\} (z_{k_{j+1}} \dots z_{k_m} \s 	z_{l_{i+1}} \dots z_{l_n})\bigg]\\
		&\ \ + z_{l_1}\bigg[\sum_{i=1}^{n} \Big\{(z_{k_1}z_{k_2} \dots z_{k_{j-1}} \s_o 	z_{l_2} \dots z_{l_i})z_{k_j}+ (z_{k_1}z_{k_2} \dots z_{k_{j-1}} \s_o z_{l_2} \dots z_{l_{i-1}})\nonumber \\
		& \ \ \ \ \times \{(1-2t)z_{k_j+l_i}+(1-\delta_{i,n}\delta_{j,m})(t^2-t)x^{k_j+l_i}\}\Big\} (z_{k_{j+1}} \dots z_{k_m} \s 	z_{l_{i+1}} \dots z_{l_n})\bigg]\\
	\end{align*}
\begin{align*}
		& \ \ +\{(1-2t)z_{k_1+l_1}+(1-\delta_{m,1}\delta_{n,1})(t^2-t)x^{k_1+l_1}\}\bigg[\sum_{i=1}^{n} \Big\{(z_{k_2} \dots z_{k_{j-1}} \s_o 	z_{l_2} \dots z_{l_i})z_{k_j} \nonumber \\
		& \ \ \ \ \ \ +(z_{k_2} \dots z_{k_{j-1}} \s_o 	z_{l_2} \dots z_{l_{i-1}}) \times \{(1-2t)z_{k_j+l_i}+(1-\delta_{i,n}\delta_{j,m})(t^2-t)x^{k_j+l_i}\}\Big\} \\
		& \hspace{3em} \times (z_{k_{j+1}} \dots z_{k_m} \s 	z_{l_{i+1}} \dots z_{l_n})\bigg] \\
		&= z_{k_1}z_{k_2} \dots z_{k_{j-1}}z_{k_j}(z_{k_{j+1}} \dots z_{k_m} \s 	z_{l_{1}} \dots z_{l_n})+\sum_{i=1}^{n}\bigg[z_{k_1}\Big\{(z_{k_2} \dots z_{k_{j-1}} \s_o 	z_{l_1} \dots z_{l_i})z_{k_j} \nonumber \\
		& \ \ \ \ \ \  +(z_{k_2} \dots z_{k_{j-1}} \s_o 	z_{l_1}z_{l_2} \dots z_{l_{i-1}}) \times \{(1-2t)z_{k_j+l_i}+(1-\delta_{i,n}\delta_{j,m})(t^2-t)x^{k_j+l_i}\}\Big\}\\
		& \ \ \ \ + z_{l_1}\Big\{(z_{k_1}z_{k_2} \dots z_{k_{j-1}} \s_o 	z_{l_2} \dots z_{l_i})z_{k_j} + (z_{k_1}z_{k_2} \dots z_{k_{j-1}} \s_o z_{l_2} \dots z_{l_{i-1}})\nonumber \\
		& \ \ \ \ \ \  \times \{(1-2t)z_{k_j+l_i}+(1-\delta_{i,n}\delta_{j,m})(t^2-t)x^{k_j+l_i}\}\Big\} \\
		& \ \ \ \ + \{(1-2t)z_{k_1+l_1}+(1-\delta_{m,1}\delta_{n,1})(t^2-t)x^{k_1+l_1}\} \Big\{(z_{k_2} \dots z_{k_{j-1}} \s_o 	z_{l_2} \dots z_{l_i})z_{k_j} \nonumber \\
		& \ \ \ \ \ \ +(z_{k_2} \dots z_{k_{j-1}} \s_o 	z_{l_2} \dots z_{l_{i-1}}) \times \{(1-2t)z_{k_j+l_i}+(1-\delta_{i,n}\delta_{j,m})(t^2-t)x^{k_j+l_i}\}\Big\} \bigg] \\
		& \hspace{3em} \times (z_{k_{j+1}} \dots z_{k_m} \s 	z_{l_{i+1}} \dots z_{l_n})\\
		&=z_{k_1}z_{k_2} \dots z_{k_{j-1}}z_{k_j}(z_{k_{j+1}} \dots z_{k_m} \s 	z_{l_{1}} \dots z_{l_n})\\
		& \ \ +\sum_{i=1}^{n}\bigg[ (z_{k_1} \dots z_{k_{j-1}} \s_o 	z_{l_1}z_{l_2} \dots z_{l_{i}})z_{k_j}+(z_{k_1} \dots z_{k_{j-1}} \s_o 	z_{l_1} z_{l_2}\dots z_{l_{i-1}})\\
		& \ \ \ \times \{(1-2t)z_{k_j+l_i}+(1-\delta_{i,n}\delta_{j,m})(t^2-t)x^{k_j+l_i}\} \bigg] \times (z_{k_{j+1}} \dots z_{k_m} \s 	z_{l_{i+1}} \dots z_{l_n})\\
		&=\sum_{i=0}^{n} \bigg[(z_{k_1}z_{k_2} \dots z_{k_{j-1}} \s_o 	z_{l_1}z_{l_2} \dots z_{l_i})z_{k_j}+ (z_{k_1}z_{k_2} \dots z_{k_{j-1}} \s_o 	z_{l_1}z_{l_2} \dots z_{l_{i-1}})\nonumber \\
		& \ \ \ \ \times \{(1-2t)z_{k_j+l_i}+(1-\delta_{i,n}\delta_{j,m})(t^2-t)x^{k_j+l_i}\}\bigg] (z_{k_{j+1}} \dots z_{k_m} \s 	z_{l_{i+1}} \dots z_{l_n}).
	\end{align*}
Hence the proof.\\
	
Proof of Theorem $2$: By the definition of t-stuffle product, we have
		\begin{align}
			\label{eq7}
			&z_m{z_p}^n \s z_u{z_p}^v \nonumber \\
			&=z_m({z_p}^n \s z_u{z_p}^v)+z_u(z_m{z_p}^n \s {z_p}^v)+\{(1-2t)z_{m+u}+(1-\delta_{n,0}\delta_{v,0})(t^2-t)x^{m+u}\}({z_p}^n \s {z_p}^v).
		\end{align}
If we apply the recursive formula \eqref{eq5} with $j=1$, we get
		\begin{align}
			\label{eq8}
			z_u{z_p}^v \s	{z_p}^n = \sum_{i=0}^{n}[{z_p}^iz_u+{z_p}^{i-1}\{(1-2t)z_{u+p}+(1-\delta_{v,0}\delta_{n,i})(t^2-t)x^{u+p}\}]({z_p}^{v} \s {z_p}^{n-i})
		\end{align}
and 
		\begin{align}
			\label{eq9}
			z_m{z_p}^n \s {z_p}^v=\sum_{i=0}^{v}[{z_p}^iz_m+{z_p}^{i-1}\{(1-2t)z_{m+p}+(1-\delta_{n,0}\delta_{v,i})(t^2-t)x^{m+p}\}]({z_p}^{n} \s {z_p}^{v-i}).
		\end{align}
Using \eqref{eq8} and \eqref{eq9} in \eqref{eq7} we can easily obtain \eqref{eq6}. Further, applying  Lemma \ref{lem2} in \eqref{eq6}, we can prove Theorem $1$. \\

\textbf{Some observations:}
Using \eqref{eq5}, we can obtain general t-stuffle product
\begin{align}
	\label{e280}
	z_{k_1}{z_p}^{r_1}z_{k_2}{z_p}^{r_2} \dots z_{k_m}{z_p}^{r_m} \s z_{l_1}{z_p}^{s_1}z_{l_2}{z_p}^{s_2} \dots z_{l_n}{z_p}^{s_n},
\end{align}
where $k_i,l_j,p$ are positive integers and $r_i, s_j$ are nonnegative integers, for $i=1, 2, \dots , m; j=1, 2, \dots , n$. If we use the formula \eqref{eq5} for $j=1$ to the above product then we can see that this product can be expressed by products
\begin{align*}
	& (U_1){z_p}^{r_1}z_{k_2}{z_p}^{r_2} \dots z_{k_m}{z_p}^{r_m} \s z_{l_1}{z_p}^{s_1}z_{l_2}{z_p}^{s_2} \dots z_{l_n}{z_p}^{s_n},\\
	& (U_2) {z_p}^{r_1}z_{k_2}{z_p}^{r_2} \dots z_{k_m}{z_p}^{r_m} \s {z_p}^{i_1}z_{l_2}{z_p}^{s_2} \dots z_{l_n}{z_p}^{s_n}, \hspace{5em} 0 \leq i_1 \leq s_1,\\
	& (U_3) {z_p}^{r_1}z_{k_2}{z_p}^{r_2} \dots z_{k_m}{z_p}^{r_m} \s {z_p}^{i_2}z_{l_3}{z_p}^{s_3} \dots z_{l_n}{z_p}^{s_n}, \hspace{5em} 0 \leq i_2 \leq s_2,\\
	& \ \ \hspace{4.5cm}  \vdots\\
	& (U_n) {z_p}^{r_1}z_{k_2}{z_p}^{r_2} \dots z_{k_m}{z_p}^{r_m} \s {z_p}^{i_{n-1}}z_{l_n}{z_p}^{s_n} , \hspace{8.5em} 0 \leq i_{n-1} \leq s_{n-1}\\
	& (U_{n+1}) {z_p}^{r_1}z_{k_2}{z_p}^{r_2} \dots z_{k_m}{z_p}^{r_m} \s {z_p}^{i_n} , \hspace{11.3em} 0 \leq i_n \leq s_{n}
\end{align*}

Let $m >1$. For $i=1, 2 , \dots  , n$, if we use the formula  \eqref{eq5} with $j=m_1+1$ to the product $(U_i)$, then we can see that $(U_i)$ can be expressed by products $(U_{i+1}), \dots , (U_{n+1})$. Also if we apply the formula \eqref{eq5} to the product $(U_{n+1})$ with $j=r_1+1$, we get
\begin{align*}
	&{z_p}^{r_1}z_{k_2}{z_p}^{r_2} \dots z_{k_m}{z_p}^{r_m} \s {z_p}^{i_n}\\
	&= \sum_{i=0}^{i_n} \bigg[({z_p}^{r_1} \s_o {z_p}^{i})z_{k_2}+ ({z_p}^{r_1}\s_o 	{z_p}^{i-1})\nonumber \\
	& \ \ \ \ \times \{(1-2t)z_{k_2+p}+(t^2-t)x^{k_2+p}\}\bigg] ({z_p}^{r_2}z_{k_3} \dots z_{k_m}{z_p}^{r_m} \s 	{z_p}^{i_n-i}).
\end{align*}
Therefore $(U_{n+1})$ can be expressed by the product with smaller $m$ of the form $(U_{n+1})$. Hence the product \eqref{e280} can be expressed by the products ${z_p}^{i} \s_o {z_p}^{j}$ and ${z_p}^{i} \s{z_p}^{j}$.

For example, consider $m=1, n=2$. Then applying the formula  \eqref{eq5} with $j=1$ to the product $z_{k_1}{z_p}^{r_1}  \s z_{l_1}{z_p}^{s_1}z_{l_2}{z_p}^{s_2}$, we get
\begin{align}
	\label{299}
&z_{k_1}{z_p}^{r_1}  \s z_{l_1}{z_p}^{s_1}z_{l_2}{z_p}^{s_2} \nonumber \\
&=z_{k_1}({z_p}^{r_1}  \s z_{l_1}{z_p}^{s_1}z_{l_2}{z_p}^{s_2})+ \{(1-2t)z_{k_1+l_1}+(t^2-t)x^{k_1+l_1}\}({z_p}^{r_1}  \s {z_p}^{s_1}z_{l_2}{z_p}^{s_2})\nonumber\\
& \ \ + \sum_{i=0}^{s_1}\big\{z_{l_1}{z_p}^{i}z_{k_1}+(1-2t)z_{l_1}{z_p}^{i-1}z_{k_1+p}+(t^2-t)z_{l_1}{z_p}^{i-1}x^{k_1+p}\big\}({z_p}^{r_1} \s {z_p}^{s_1-i}z_{l_2}{z_p}^{s_2} )\nonumber \\
& \ \ \ +\big\{(1-2t)z_{l_1}{z_p}^{s_1}z_{k_1+l_2}+(1-\delta_{r_1,0}\delta_{s_2,0})(t^2-t)z_{l_1}{z_p}^{s_1}x^{k_1+l_2}\big\}
({z_p}^{r_1} \s {z_p}^{s_2})\nonumber\\
& \ \ \ + \sum_{i=0}^{s_2}\big\{z_{l_1}{z_p}^{s_1}z_{l_2}{z_p}^{i}z_{k_1} +(1-2t)z_{l_1}{z_p}^{s_1}z_{l_2}{z_p}^{i-1}z_{k_1+p}+(1-\delta_{r_1,0}\delta_{s_2,i})\nonumber \\
& \ \ \ \ \times (t^2-t)z_{l_1}{z_p}^{s_1}z_{l_2}{z_p}^{i-1}x^{k_1+p}\big\}({z_p}^{r_1} \s {z_p}^{s_2-i} ).
\end{align}
Therefore, we have to compute the products 
\begin{align*}
	&(i) {z_p}^{r}  \s z_{l_1}{z_p}^{s_1}z_{l_2}{z_p}^{s_2}\\
	&(ii) {z_p}^{r}  \s {z_p}^{s_1}z_{l}{z_p}^{s_2} 
\end{align*}
Applying the recursive formula \eqref{eq5} with $j=1$ to the product $(i)$, we get
\begin{align}
	\label{3012}
	&{z_p}^{r}  \s z_{l_1}{z_p}^{s_1}z_{l_2}{z_p}^{s_2} \nonumber\\
	&=\sum_{i=0}^{r}\big\{{z_p}^{i}z_{l_1}+(1-2t){z_p}^{i-1}z_{l_1+p}+(t^2-t){z_p}^{i-1}x^{l_1+p}\big\}({z_p}^{r-i} \s {z_p}^{s_1-i}z_{l_2}{z_p}^{s_2} ).
\end{align}
Thus the product $(i)$ can be expressed by the product $(ii)$. Now applying the recursive formula  \eqref{eq5} with $j=s_1+1$ to the product $(ii)$, we get
\begin{align}
	\label{311}
	&{z_p}^{r}  \s {z_p}^{s_1}z_{l}{z_p}^{s_2}\nonumber \\
	&= \sum_{i=0}^{r}\big\{({z_p}^{i} \s_0 {z_p}^{s_1})z_l  +({z_p}^{i-1} \s_0 {z_p}^{s_1})\{(1-2t)z_{l+p}+(1-\delta_{i,r}\delta_{s_2,0})(t^2-t)x^{l+p}\}\big\}({z_p}^{r-i} \s {z_p}^{s_2} ).
\end{align}
Using \eqref{3012} and \eqref{311} in \eqref{299} we get an expression for the product $z_{k_1}{z_p}^{r_1}  \s z_{l_1}{z_p}^{s_1}z_{l_2}{z_p}^{s_2}$ in terms of simple products ${z_p}^{i} \s{z_p}^{j}$ and ${z_p}^{i} \s_o {z_p}^{j}$.\\ 
\textbf{Note:} ${z_p}^{i} \s_o {z_p}^{j}$ can be obtained in a similar way of obtaining ${z_p}^{i} \s{z_p}^{j}$ in Lemma $1$.

\section{Some Applications}
Recall, for positive integers $d$ and $e$, 
\begin{align*}
	\mathbb{F}_d^e=\sum_{\substack{a_1 + \dots +a_{d}=ep;\\a_l=rp,r \in \mathbb{N}; |A|= j} } z_{a_1} \dots z_{a_{d}}.
\end{align*}
Using this notation \eqref{equn1} can be rewritten as 
\begin{align}
	\label{u32}
	&\hspace{1em}{z_p}^m \s {z_p}^n \nonumber \\
	&=\sum_{0 \leq k \leq {n}}\binom{m+n-2k}{m-k} \sum_{\substack{ i+j=k; i,j \in \mathbb{Z}_{\geq 0}} } {(t^2-t)}^i{(1-2t)}^j ~\mathbb{F}_{m+n-i-k}^{m+n}.	
\end{align} 

Now we prove the following result.

\begin{proposition}
	For positive integers $p\text{ and }k$, we have
	\begin{align}
		\label{27}
		&\sum_{\substack{m+n=k,\\m,n \geq 0}}{(-1)}^m z_p^m \s z_p^n \nonumber\\ & =
		\begin{cases}
			0& \text{if $k$ is odd}\\
			{(-1)}^\frac{k}{2}\sum_{\substack{l_1+l_2=\frac{k}{2};\\ l_1,l_2 \in \mathbb{Z}_{\geq 0}}}  {(t^2-t)}^{l_1}{(1-2t)}^{l_2} ~\mathbb{F}_{l_2}^k & \text{if $k$ is even} 
		\end{cases}
	\end{align}
\end{proposition}

\begin{proof}
	For odd $k$, we have
	\begin{align*}
		&\sum_{\substack{m+n=k,\\m,n \geq 0}}{(-1)}^m z_p^m \s z_p^n \\
		&={z_p}^k-{z_p} \s {z_p}^{k-1}+ {z_p}^2 \s {z_p}^{k-2}- \dots +{(-1)}^{\frac{k-1}{2}}{z_p}^{\frac{k-1}{2}} \s {z_p}^{\frac{k+1}{2}} \\
		& \ \ \ + {(-1)}^{\frac{k+1}{2}}{z_p}^{\frac{k+1}{2}} \s {z_p}^{\frac{k-1}{2}}+ \dots -{z_p}^{k-2} \s {z_p}^{2}+{z_p}^{k-1} \s z_p-{z_p}^k\\
		&=0. 
	\end{align*}
 Suppose $k$ is even. Then
we have,
	\begin{align*}
		&\sum_{\substack{m+n=k,~ k:\text{even}\\m,n \geq 1}}{(-1)}^m z_p^m \s z_p^n \\
		&=2 \sum_{d=0}^{\frac{k}{2}-1} {(-1)}^d z_p^d \s z_p^{k-d}+ {(-1)}^\frac{k}{2} z_p^\frac{k}{2} \s z_p^{\frac{k}{2}} \\
		&=2 \sum_{d=0}^{\frac{k}{2}-1} {(-1)}^d \sum_{\substack{0 \leq l \leq d; \\l_1+l_2=l; l_1,l_2 \in \mathbb{Z}_{\geq 0}} } \binom{k-2l}{d-l}{(t^2-t)}^{l_1}{(1-2t)}^{l_2} ~\mathbb{F}_{k-l-l_1}^k\\
		& \ \ + {(-1)}^\frac{k}{2} \sum_{\substack{0 \leq l \leq \frac{k}{2}; \\l_1+l_2=l; l_1,l_2 \in \mathbb{Z}_{\geq 0}} } \binom{k-2l}{\frac{k}{2}-l}{(t^2-t)}^{l_1}{(1-2t)}^{l_2} ~\mathbb{F}_{k-l-l_1}^k \hspace{1cm} [using ~\eqref{u32}]\\
		&=\Bigg[-2\sum_{l=0}^{1}\binom{k-2l}{1-l}+2\sum_{l=0}^{2}\binom{k-2l}{2-l}+ \dots +2{(-1)}^{\frac{k}{2}-1}\sum_{l=0}^{\frac{k}{2}-1}\binom{k-2l}{\frac{k}{2}-1-l}\\
		& \ \ \ +{(-1)}^{\frac{k}{2}}\sum_{l=0}^{\frac{k}{2}-1}\binom{k-2l}{\frac{k}{2}-l}\Bigg] \times \sum_{\substack{l_1+l_2=l;\\ l_1,l_2 \in \mathbb{Z}_{\geq 0}} } {(t^2-t)}^{l_1}{(1-2t)}^{l_2} ~\mathbb{F}_{k-l-l_1}^k + 2{z_p}^k\\
		& \ \ \ + {(-1)}^\frac{k}{2}\sum_{\substack{l_1+l_2=\frac{k}{2};\\ l_1,l_2 \in \mathbb{Z}_{\geq 0}}}  {(t^2-t)}^{l_1}{(1-2t)}^{l_2} ~\mathbb{F}_{\frac{k}{2}-l_1}^k\\
		&=\bigg\{2\sum_{d=0}^{\frac{k}{2}-1} {(-1)}^d \binom{k}{d}+  {(-1)}^\frac{k}{2}\binom{k}{\frac{k}{2}}\bigg\} {z_p}^k\\
		& \ \ \
		+ \sum_{l=1}^{\frac{k}{2}-1}{(-1)}^l \Bigg[\sum_{d=0}^{\frac{k}{2}-l-1}\biggl\{2{(-1)}^d\binom{k-2l}{d}\biggr\} +{(-1)}^{\frac{k}{2}-l}\binom{k-2l}{\frac{k}{2}-l}\Bigg] \\
		& \ \ \ \times \sum_{\substack{l_1+l_2=l;\\ l_1,l_2 \in \mathbb{Z}_{\geq 0}} } {(t^2-t)}^{l_1}{(1-2t)}^{l_2} ~\mathbb{F}_{k-l-l_1}^k + {(-1)}^\frac{k}{2}\sum_{\substack{l_1+l_2=\frac{k}{2};\\ l_1,l_2 \in \mathbb{Z}_{\geq 0}}}  {(t^2-t)}^{l_1}{(1-2t)}^{l_2} ~\mathbb{F}_{l_2}^k \\
		& = \biggl\{2{(-1)}^{\frac{k}{2}-1}\binom{k-1}{\frac{k}{2}-1}+ {(-1)}^{\frac{k}{2}}\binom{k}{\frac{k}{2}}\biggr\} {z_p}^k\\
	\end{align*}
\begin{align*}
		& \ \ + \sum_{l=1}^{\frac{k}{2}-1}{(-1)}^l \Bigg[ 2{(-1)}^{\frac{k}{2}-l-1}\binom{k-2l-1}{\frac{k}{2}-l-1}+ {(-1)}^{\frac{k}{2}-l}\binom{k-2l}{\frac{k}{2}-l}\Bigg] \\
		& \ \ \ \times \sum_{\substack{l_1+l_2=l;\\ l_1,l_2 \in \mathbb{Z}_{\geq 0}} } {(t^2-t)}^{l_1}{(1-2t)}^{l_2} ~\mathbb{F}_{k-l-l_1}^k  + {(-1)}^\frac{k}{2} \sum_{\substack{l_1+l_2=\frac{k}{2};\\ l_1,l_2 \in \mathbb{Z}_{\geq 0}}}  {(t^2-t)}^{l_1}{(1-2t)}^{l_2} ~\mathbb{F}_{l_2}^k.
	\end{align*}

Using the relation $\binom{k-1}{\frac{k}{2}-1}=2\binom{k}{\frac{k}{2}}$ and $\binom{k-2l-1}{\frac{k}{2}-l-1}=2\binom{k-2l}{\frac{k}{2}-l}$ in the above equation, we get the desired result. 

	
\end{proof}
\begin{remark}
	 For $t=0$, \eqref{27} gives the following result on multiple zeta values, which is obtained in $[6]$,
	\begin{align*}
	\sum_{\substack{m+n=k,\\m,n \geq 0}}{(-1)}^m z_p^m \sh z_p^n  =
	\begin{cases}
		0& \text{if $k$ is odd}\\
		{(-1)}^{\frac{k}{2}}{z_{2p}}^\frac{k}{2} & \text{if $k$ is even}.
	\end{cases}
\end{align*}
\end{remark} 

Applying the $\mathbb{Q}[t]$-linear map $Z^t$ on both sides of \eqref{27}, we get

\begin{corollary}
	For positive integers $p \geq 2, k \geq 1$, we have
	\begin{align}
		\label{28}
		&\sum_{\substack{m+n=k,~ k:\text{ even}\\m,n \geq 0}}{(-1)}^m \z^t(\{p\}^m) \z^t(\{p\}^n) \nonumber \\
		&={(-1)}^{\frac{k}{2}}\sum_{\substack{ i+j=\frac{k}{2}; i,j \in \mathbb{Z}_{\geq 0}\\a_1 + \dots +a_{{j}}=kp;\\a_l=2rp;r \in \mathbb{N}}} {(t^2-t)}^i{(1-2t)}^j \z^t(a_1, \dots, a_{j}).
	\end{align}
\end{corollary} 
For $t=0$, \eqref{28} gives the following relation of multiple zeta values, which is  obtained in \cite{6},
\begin{align}
	\label{29}
	&\sum_{\substack{m+n=k,~ k:\text{ even}\\m,n \geq 0}}{(-1)}^m \z(\{p\}^m) \z(\{p\}^n) ={(-1)}^{\frac{k}{2}}\z(\{2p\}^\frac{k}{2}),
	\end{align}
here taking $\frac{k}{2}=l$, we get
 \begin{align}
 	\label{291}
 	&\z(\{2p\}^l) = \sum_{m=0}^{2l}{(-1)}^{m+l} \z(\{p\}^m) \z(\{p\}^{2l-m}).
 	 \end{align}
Also, for $t=1$, \eqref{28} gives
\begin{align}
	\label{30}
	&\sum_{\substack{m+n=k,~ k:\text{ even}\\m,n \geq 0}}{(-1)}^m \z^\star(\{p\}^m) \z^\star(\{p\}^n) =\z^\star(\{2p\}^\frac{k}{2}).
\end{align}
If we take $p=2$ in \eqref{29}, and then using the formulas $$\z(\{2\}^k)=\dfrac{\pi^{2k}}{(2k+1)!};~ \z(\{4\}^k)=\dfrac{2^{2k+1}\pi^{4k}}{(4k+2)!},$$ 
we get the following relation of binomial coefficients
\begin{align}
	\label{31}
	\sum_{\substack{m+n=k,~ k:\text{ even}\\m,n \geq 0}}{(-1)}^m \dfrac{1}{(2m+1)!(2n+1)!}={(-1)}^{\frac{k}{2}}\dfrac{2^{k+1}}{(2k+2)!}.
\end{align}
Putting $p=4$ in \eqref{291}, for any positive integer $l$, we get
\begin{align*}
	\z(\{8\}^l)&= \sum_{m=0}^{2l}{(-1)}^{m+l} \z(\{4\}^m) \z(\{4\}^{2l-m})\\
	&=\sum_{m=0}^{2l}{(-1)}^{m+l}\dfrac{2^{4l+2}\pi^{8l}}{(4m+2)!(8l-4m+2)!}.
\end{align*} 
From \cite[ Eq.(3.26)]{4}, we have,
\begin{align*}
	\z(\{8\}^l)={(-1)}^l\sum_{\substack{ n_0+n_1+n_2+n_3=4l\\ n_0, \dots, n_3 \geq 0 }}\dfrac{{(\sqrt{-1})}^{n_1+2n_2+3n_3}}{(2n_0+1)! \dots (2n_3+1)!} \pi^{8l}.
\end{align*}
Comparing the above two formulas for $	\z(\{8\}^l)$, we get
\begin{align*}
	&\sum_{\substack{ n_0+n_1+n_2+n_3=4l\\ n_0, \dots, n_3 \geq 0 }}\dfrac{{(\sqrt{-1})}^{n_1+2n_2+3n_3}}{(2n_0+1)! \dots (2n_3+1)!}=\sum_{m=0}^{2l}{(-1)}^{m}\dfrac{2^{4l+2}}{(4m+2)!(8l-4m+2)!}.
\end{align*}	\\

$\mathbf{Acknowledgment}$: The research of first author is supported by the University Grants Commission (UGC), India through NET-JRF (Ref. No. 191620198830).\\



	{}

\end{document}